\documentclass[12pt]{amsart}

\usepackage{parskip}
\usepackage{amsmath}
\usepackage{mathtools,amssymb,amsthm}
\usepackage{hyperref}
\usepackage{fancyref}
\usepackage{cleveref}

\newtheorem{thm}{Theorem}
\newtheorem{prop}[thm]{Proposition}
\newtheorem{question}[thm]{Question}

\newtheorem{lem}[thm]{Lemma}
\newtheorem{cor}[thm]{Corollary}
\theoremstyle{definition}

\theoremstyle{definition}

\theoremstyle{definition}

\theoremstyle{remark}

\theoremstyle{definition}
\newtheorem*{defn*}{Definition}
\newtheorem*{acknowledgements}{Acknowledgements}

\newtheorem*{theorem1}{Theorem 1}
\newtheorem*{thm*}{Theorem}
\newtheorem*{lem*}{Lemma}

\newcommand{\Addresses}{{
  \bigskip
  \footnotesize

  \textit{Contact E-mail address}, A.~F.~J.-C.~Launois: \texttt{A.Launois@outlook.com}}}

\title{On a $q$-Skew Amitsur's Theorem}
\author{Aristide F. J.-C. Launois}
\date{April 2026}

\begin{document}
\bibliographystyle{amsalpha}

\begin{abstract}
    Let $R$ be an algebra over an uncountable field, $\sigma$ a locally torsion automorphism and $\delta$ a locally nilpotent left $\sigma$-derivation such that $q\sigma\delta = \delta\sigma$, where $q$ is a nonzero scalar. We show that the constant part of the Jacobson radical of the Ore extension $R[x;\sigma,\delta]$ is nil. This partially answers a question of Greenfeld, Smoktunowicz and Ziembowski posed in 2019. As a corollary, we employ Shin's 2024 result to prove a q-skew Amitsur's theorem whenever the field is additionally assumed to be of characteristic zero. That is, the Jacobson radical of $R[x;\sigma, \delta]$ is $N[x;\sigma,\delta]$ for some nil ideal $N$ of $R$. 
\end{abstract}

\maketitle

\section{Introduction}

Let $R$ be a ring. By ring we shall mean a noncommutative associative ring not necessarily with 1. By ideal we will mean a two-sided ideal unless otherwise stated. Given an automorphism $\sigma:R \to R$, a \textbf{$\sigma$-derivation} is an additive map $\delta:R\to R$ satisfying the twisted Leibniz rule,\[
        \delta(ab) = \sigma(a)\delta(b) + \delta(a)b,
    \]for all $a,b \in R$. Throughout this paper, $\sigma:R \to R$ will denote an automorphism and $\delta:R\to R$ a $\sigma$-derivation. Given $R$, $\sigma$, and $\delta$, we can form the \textbf{Ore extension} $R[x;\sigma, \delta]$ which is a free left $R$-module with basis $1,x,x^2,\dots$ with multiplication rule defined by \[
        xr = \sigma(r)x + \delta(r)
    \] for all $r \in R$. Since $R[x;\sigma,\delta]$ is a free left $R$-module, we can define the (left) \textbf{degree} of an element $$f = a_0+a_1x+\cdots +a_nx^n \in R[x;\sigma,\delta]$$ to be the largest $i$ such that $a_i \neq 0$. We denote the degree by $\deg(f)$.

If, for every $a \in R$, there exists an $n = n(a) \in \mathbb{N}$ such that $\delta^n(a) = 0$, we say that $\delta$ is \textbf{locally nilpotent}. If, for every $a \in R$, there exists an $n = n(a) \in \mathbb{N}$ such that $\sigma^n(a) = a$, we say that $\sigma$ is  \textbf{locally torsion}. Note that when $R$ is an algebra over a field $k$ we assume that the derivation and automorphism are $k$-linear.

There are many important objects arising in various parts of mathematics that can be described using (iterated) Ore extensions. For example, universal enveloping algebras of some Lie algebras, quantum groups, and rings of differential operators, see for instance \cite[Prologue]{goodearl2004introduction}. The Jacobson radical is the most well studied radical and allows us to understand the structure and representation theory of rings and algebras. In 1956, Amitsur proved that the Jacobson radical of a polynomial ring $J(R[x]) = I[x]$, where $I = J(R[x])\cap R$ is a nil ideal of $R$ \cite{Amitsur_1956}. Whilst the result is very powerful, it does not tell us what the nil ideal is. It is natural to ask if it is the largest nil ideal of the ring, the upper nilradical $\text{Nil}^*(R)$. It turns out that this is equivalent to another famous problem in ring theory \cite[Exercise 10.25]{lam1991first}, the Koethe conjecture. In its original form it asks whether the sum of two left nil ideals is also nil \cite{Kothe1930} which was first posed almost a century ago. In 1972, Krempa showed that the problem is equivalent to whether polynomial rings over nil rings are Jacobson radical \cite[Theorem 2]{Krempa1972}. In 2000, Smoktunowicz showed that polynomial rings over nil rings need not be nil \cite{SMOKTUNOWICZ2000427}. This strengthened suspicions that the Koethe conjecture may be answered in the negative in general. It is known to hold for Noetherian and P.I. rings. 

Motivated by understanding the structure of Ore extensions and the Koethe conjecture, there has been work done on generalising Amitsur's result to Ore extensions. In 1980, Bedi and Ram showed that, when $\delta =0$ and $\sigma$ is locally torsion, a `skew' Amitsur's theorem holds \cite[Theorem 3.1]{Bedi1980}. In the same paper they prove that, in general, $J(R[x;\sigma]) \neq (J(R[x;\sigma])\cap R)[x;\sigma]$ \cite[Example 3.5]{Bedi1980}. However, they gave a nice characterisation of the Jacobson radical in this case. On the contrary, in 1983, Ferrero, Kishimoto, and Motose showed that, for $\sigma = 1$, the equality $J(R[x;\delta]) = (J(R[x;\delta])\cap R)[x;\delta]$ holds \cite[Theorem 3.2]{DiffPolyRadFerrero}. In the same paper it was shown that $J(R[x;\delta])\cap R$ is nil if $R$ is commutative \cite[Theorem 3.4]{DiffPolyRadFerrero}. This was followed up to prove a skew Amitsur theorem in 2017 where Smoktunowicz showed that $J(R[x;\delta])\cap R$ is nil whenever $\delta$ is locally nilpotent and $R$ is an algebra over a field of characteristic $p >0$ \cite[Theorem 2]{AgataHowFarCanWeGo}. In general, however, we cannot expect $J(R[x;\delta])\cap R$ to be nil, as was shown by Smoktunowicz \cite[Theorem 1]{AgataHowFarCanWeGo}. In 2014, Smoktunowicz and Ziembowski showed that $J(R[x;\delta])\cap R$ is nil whenever $\delta$ is locally nilpotent and $R$ is an algebra over an uncountable field \cite[Proposition 2.1]{SMOKTUNOWICZuncountablefield}. 

Recently, in 2024, Shin extended the results of Bedi and Ram to give a characterisation of the Jacobson radical when $R$ is an algebra over a field $k$ of characteristic zero and $\delta$ is \textbf{$q$-skew} \cite[Theorem 1]{shin2024jacobsonradicalsoreextensions}, that is, $q\sigma\delta = \delta\sigma$ for some $q \in k\backslash \{0\}$. In the same paper it was shown that, for an algebra over a field of characteristic $p>0$, a locally torsion automorphism $\sigma$, and a locally nilpotent 1-skew derivation $\delta$, the ideal $J(R[x;\sigma,\delta])\cap R$ is nil \cite[Theorem 3]{shin2024jacobsonradicalsoreextensions}. However, whether or not a q-skew Amitsur's theorem holds in the case of a general Ore extension remains open. 
    
We will prove the following result following the ideas of Smoktunowicz and Ziembowski in \cite[Proposition 2.1]{SMOKTUNOWICZuncountablefield}.

\begin{thm}
\label{theorem}
    Let $R$ be a $k$-algebra where $k$ is an uncountable field, $\sigma$ a locally torsion $k$-algebra automorphism, and $\delta$ a $k$-linear $\sigma$-derivation which is locally nilpotent and $q$-skew. Then $$I = \{a\in R:ax^{t}\in J(R[x;\sigma,\delta])\}$$ is a nil ideal of $R$ for any $t \in \mathbb{N}$.
\end{thm}

For $t = 0$, this result gives a partial answer to the following open question posed by Greenfeld, Smoktunowicz and Ziembowski in 2019 \cite{GreenfieldSmoktunowiczZiembowski}. 

\begin{question}\cite[Question 6.17]{GreenfieldSmoktunowiczZiembowski}
    Consider $J(R[x;\sigma,\delta]) \cap R$. Is it nil if we assume that $\delta$ is locally nilpotent? What if we assume that $\sigma$ is locally torsion?
\end{question}

As a corollary, we will show that a $q$-skew Amitsur's theorem holds whenever the field is additionally assumed to be of characteristic zero.

\section{Supporting Lemmas}

Recall that an element $r\in R$ is \textbf{quasi-invertible} if there exists an $s \in R$ such that \[
        r+s+rs = 0.
    \]We say that $s$ is the \textbf{quasi-inverse} of $r$. Since any ring can be embedded into a ring $R^1$ with unit, this can be equivalently written as $(1+r)(1+s) =1$ in $R^1$. It follows that quasi-inverses are unique.

\begin{lem}
    \label{UniquenessofQuasiInverse}
    Let $s,s'\in R$ be quasi-inverses of $r\in R$. Then $s=s'$.
\end{lem}

\begin{proof}
    By the previous remark, in $R^1$\[
        (1+s)(1+r) = 1 = (1+r)(1+s').
    \]Thus,\[
         1+s = (1+s)(1+r)(1+s') = 1+s'.
    \]Hence, $s=s'$ as required.
\end{proof}

It is well known that every element of the Jacobson radical of a ring is quasi-invertible, see for instance \cite[Exercise 4.4]{lam1991first}.

In \cite{GreenfieldSmoktunowiczZiembowski}, Greenfeld, Smoktunowicz and Ziembowski showed that the Ore power series ring $R[[x;\sigma,\delta]]$ is well defined whenever the derivation is locally nilpotent. 

\begin{thm*} \cite[Theorem 5.3]{GreenfieldSmoktunowiczZiembowski}
\label{OrePowerSeries}
    Let $R$ be a ring, $\sigma$ an endomorphism of $R$ and $\delta$ a $\sigma$-derivation that is locally nilpotent. Then the natural extension of multiplication from $R[x;\sigma,\delta]$ to $R[[x;\sigma,\delta]]$ is well defined.
\end{thm*}

\section{A $q$-Skew Amitsur's Theorem}

Recall that a nil ideal is one in which every element is nilpotent. In \cite{Amitsur_1956}, Amitsur proved the following well-known theorem.

\begin{thm*}(Amitsur's Theorem)
\label{amitsurs}
    Let $R$ be a ring. Then $$J(R[x]) = (J(R[x])\cap R)[x].$$ Moreover, $J(R[x])\cap R$ is a nil ideal of $R$.
\end{thm*}

We would like to generalise this to Ore extensions, that is,

\begin{enumerate}
    \item $J(R[x;\sigma,\delta]) = (J(R[x;\sigma,\delta])\cap R)[x;\sigma,\delta]$;
    \item and $J(R[x,\sigma,\delta])\cap R$ is a nil ideal of $R$.
\end{enumerate}
We call (1) `Amitsur's property'. If both (1) \& (2) hold we say that a skew Amitsur's theorem holds.

In order for the above to even be well-defined, we note the following result of Shin. Recall that we say $\delta$ is a \textbf{$q$-skew} $\sigma$-derivation of an algebra over a field $k$ if there exists a $q\in k\backslash\{0\}$ such that $q\sigma\delta = \delta\sigma$.

\begin{lem*}\cite[Lemma 1 \& 2]{shin2024jacobsonradicalsoreextensions}
    Let $R[x;\sigma,\delta]$ be an Ore extension such that $\sigma$ is an automorphism and $\delta$ is $q$-skew. Then $N = J(R[x;\sigma,\delta])\cap R$ is a $(\sigma,\delta)$-stable ideal. That is, $\sigma(N)\subseteq N$ and $\delta(N)\subseteq N$.
\end{lem*}

We will need the following useful formulas that can be found in \cite[Section 2.5]{Goodearlformula}.

\begin{prop}\cite[Section 2.5]{Goodearlformula}
\label{GoodearlFormula}
    Let $\sigma$ be an endomorphism of $R$ and $\delta$ a $q$-skew $\sigma$-derivation on $R$. Then, for any $a,b \in R$ and $n \in \mathbb{N}$,\[
        x^na = \sum_{i=0}^n \binom{n}{i}_q \sigma^i\delta^{n-i}(a)x^i
    \]
    and\[
        \delta^n(ab)=\sum_{i=0}^n\binom{n}{i}_q \sigma^{n-i}\delta^i(a)\delta^{n-i}(b),
    \] where $\binom{n}{i}_q$ is the evaluation at $t=q$ of the polynomial function \[
        \binom{n}{i}_t = \frac{(t^n-1)(t^{n-1}-1)\cdots(t-1)}{(t^i-1)(t^{i-1}-1)\cdots(t-1)(t^{n-i}-1)(t^{n-i-1}-1)\cdots(t-1)},
    \]called a \textbf{$q$-binomial coefficient}. Note that, for $q=1$, we recover the usual binomial coefficients.
\end{prop}

Iterating this formula we obtain the following.

\begin{lem}
    \label{iteratedgoodearlformula}Let $R$ be a $k$-algebra where $k$ is a field, $\sigma$ a $k$-algebra automorphism, and $\delta$ a $k$-linear
    $\sigma$-derivation which is $q$-skew. Then, for any $a \in R$ and $n,m \in \mathbb{N}$, we have\[
    (ax^n)^m = \sum_{i_1,\dots,i_{m-1} = 0}^n aF_{i_{m-1}}(aF_{i_{m-2}}(\cdots aF_{i_1}(a))))x^{i_1+i_2+\cdots+i_{m-1}+n},
    \] where we define the map $F_i:R \to R$ by \[
        F_i(a):= \binom{n}{i}_q\sigma^i\delta^{n-i}(a).
    \]
\end{lem}

\begin{proof}
    We proceed by induction on $m$. For $m=1$, the result follows immediately by Proposition \ref{GoodearlFormula}. Similarly, by Proposition \ref{GoodearlFormula}\begin{align*}
        &(ax^n)^{m+1} = (ax^n)(ax^n)^m \\&=  ax^n\sum_{i_1,\dots,i_{m-1} = 0}^n aF_{i_{m-1}}(aF_{i_{m-2}}(aF_{i_{m-3}}(\cdots aF_{i_1}(a))))x^{i_1+i_2+\cdots+i_{m-1}+n}\\&= \sum^n_{i_1,\dots,i_{m-1},i_m = 0}aF_{i_{m}}(aF_{i_{m-1}}(aF_{i_{m-2}}(\cdots aF_{i_1}(a))))x^{i_1+i_2+\cdots+i_{m}+n}.
    \end{align*}
\end{proof}

We also notice some bounds on the nilpotency degree of elements of $R$. For any $a \in R$ and a locally nilpotent derivation $\delta$ on $R$, we denote by $\text{nildeg}(a) = N$ the number such that $\delta^N(a) = 0$ and $\delta^{N-1}(a) \neq 0$. 

\begin{lem}
    \label{nildeglem}
    Let $R$ be a $k$-algebra where $k$ is a field, $\sigma$ a $k$-algebra automorphism, and $\delta$ a $k$-linear
    $\sigma$-derivation which is locally nilpotent and $q$-skew. Let $a,b \in R$ and $N_a = \text{nildeg}(a), N_b = \text{nildeg}(b)$. Then \begin{enumerate}
        \item \label{nildeg1}$\text{nildeg}(a+b)\leq \max\{N_a,N_b\}$;
        \item\label{nildeg3}$\text{nildeg}(\lambda\sigma^s\delta^d(a))\leq\max\{0,N_a-d\}$;
        \item\label{nildeg2}$\text{nildeg}(ab)\leq N_a+N_b  $;
    \end{enumerate}
    for any $\lambda\in k$ and $s,d \in \mathbb{N}$.
\end{lem}

\begin{proof}
\begin{enumerate}
    \item 

    Firstly,\[
        \delta^{\max\{N_a,N_b\}}(a+b) = \delta^{\max\{N_a,N_b\}}(a)+\delta^{\max\{N_a,N_b\}}(b) = 0,
    \] which proves the inequality.

    \item We only need to show that $\text{nildeg}(\sigma(a))\leq N_a$ and similarly $\text{nildeg}(\delta(a))\leq N_a-1$ so that the result follows by induction on the degree of $\lambda\sigma^s\delta^d$. Note that \[
        \delta^{N_a-1}(\delta(a)) = \delta^{N_a}(a)=0.
    \] Moreover,\[
        \delta^{N_a}(\sigma(a)) = \delta^{N_a-1}(q\sigma\delta(a)) = q \delta^{N_a-1}(\sigma\delta(a)) = q^{N_a}\sigma(\delta^{N_a}(a)) = 0,
    \]as required.

    \item By the product formula in Proposition \ref{GoodearlFormula},\[
        \delta^{N_a+N_b}(ab) = \sum_{i=0}^{N_a+N_b}\binom{N_a+N_b}{i}_q\sigma^{N_a+N_b-i}\delta^i(a)\delta^{N_a+N_b-i}(b) = 0,
    \] which vanishes since for all $i$ we either have that $i \geq N_a$ or we have that $N_a+N_b-i \geq N_b$. 

    \end{enumerate}
    
\end{proof}

For $f \in R[x;\sigma,\delta]$, let $(f)_k$ denote the (left) coefficient at $x^k$ in $f$.

\begin{lem}
    \label{mexists}  Let $R$ be a $k$-algebra where $k$ is a field, $\sigma$ a $k$-algebra automorphism, and $\delta$ a $k$-linear
    $\sigma$-derivation which is locally nilpotent and $q$-skew. Let $a \in R$ and $n$ be some large natural number such that $\sum_{i=1}^\infty (ax^n)^i$ has finite degree $d$ and $n > \text{nildeg}(a)$. Then for all sufficiently large $m$ we have that \[
        ((ax^n)^m)_{dn} = 0.
    \]
\end{lem}

\begin{proof}
    By Lemma \ref{iteratedgoodearlformula}, we have that \[
        ((ax^n)^m)_{dn} = \sum_{\substack{i_1,\dots,i_{m-1} = 0\\i_1+i_2+\cdots+i_{m-1} +n = dn}}^n aF_{i_{m-1}}(aF_{i_{m-2}}(aF_{i_{m-3}}(\cdots aF_{i_1}(a)))).
    \]
    In particular, the indices in the sum satisfy\[
        i_1+i_2+\cdots+i_{m-1} = (d-1)n.
    \]
    
    Suppose that $((ax^n)^m)_{dn} \neq 0$. Then at least one term in the sum must be nonzero. We show that $m$ is uniformly bounded from above which will prove our claim. 
    
    Let $N = \text{nildeg}(a)$ and put $$N_j := \text{nildeg}(aF_{i_{j}}(aF_{i_{j-1}}(aF_{i_{j-2}}(\cdots aF_{i_1}(a)))))$$ for $1 \leq j \leq m-1$ with $N_0 := N$. By Lemma \ref{nildeglem}, we have that \begin{align*}
        N_j &\leq N + \text{nildeg}(F_{i_{j}}(aF_{i_{j-1}}(\cdots aF_{i_1}(a)))) \\&\leq N+N_{j-1}-(n-i_j) \\&=N+N_{j-1}-n+i_{j}.\end{align*}

    It follows by iterating that \begin{align*}
        N_j &\leq N-n+i_j+N_{j-1} \\&\leq N-n+i_j +N-n+i_{j-1}+N_{j-2}\\&= (j+1)N-jn+\sum_{k=1}^ji_k\end{align*}for all $1\leq j \leq m-1$.

    Noting that $N_j \geq 0$ for all $j$, we have \begin{align*}
        0\leq N_{m-1}&\leq mN-(m-1)n+\sum_{k=1}^{m-1}i_k \\&= mN- (m-1)n+(d-1)n \\&= -m(n-N)+dn.
    \end{align*} Thus, since $n-N > 0$ by assumption,\[
        m(n-N)\leq dn \implies m \leq \frac{dn}{n-N}, 
    \] as required. 
\end{proof}

We are now ready to begin the proof of Theorem \ref{theorem}. The proof follows the ideas of \cite[Proposition 2.1]{SMOKTUNOWICZuncountablefield}.

\begin{lem}
\label{lemma}
    Let $R$ be a $k$-algebra where $k$ is an uncountable field, $\sigma$ a $k$-algebra automorphism, and $\delta$ a $k$-linear
    $\sigma$-derivation which is locally nilpotent and $q$-skew. Let $a \in R$ such that there exists $t \in \mathbb{N}$ with $ax^{t} \in J(R[x;\sigma,\delta])$. Then there exists some $d \in \mathbb{N}$ such that $a\sigma^n(a)\sigma^{2n}(a)\cdots\sigma^{(d-1)n}(a) = 0$ for any $n \geq t$ such that $n > \text{nildeg}(a)$.
\end{lem}

\begin{proof}
     Fix an arbitrarily large natural number $n \geq t$ with $n > \text{nildeg(a)}$. Then since $ax^{t} \in J(R[x;\sigma,\delta])$ so is $ax^n$. Thus, $ax^n$ has quasi-inverse in $R[x;\sigma,\delta]$. Note that it has a quasi-inverse in $R[[x;\sigma,\delta]]$ given by  \[
        f = \sum_{i=1}^\infty (-ax^n)^i.
    \] So, by Lemma \ref{UniquenessofQuasiInverse}, these two inverses are equal. Thus, $$f \in R+Rx+Rx^2+\cdots+Rx^{d_1-1},$$ for some $d_1$. 
    
    Similarly, we obtain that the quasi-inverses of $\lambda a x^n$ are $f_\lambda$, where $$f_\lambda = \sum_{i=1}^\infty (-\lambda ax^n)^i \in R+Rx+Rx^2+\cdots+Rx^{d_\lambda-1},$$ for some $d_\lambda$ and all $\lambda \in k$. Then, considering \[
        k = \bigcup_{d=0}^\infty \{\lambda \in k:\deg(f_\lambda) = d-1\}
    \] and that $k$ is uncountable it follows that, for some $d$, we have infinitely many $\lambda \in k$ such that $f_\lambda \in R+Rx+\cdots+Rx^{d-1}$. Hence, the $x^{dn}$ coefficient of $f_\lambda$ is zero for infinitely many $\lambda \in k$.
    
    Note that \[
        (f_\lambda)_{dn} = ((-\lambda ax^n)^d)_{dn}+((-\lambda ax^n)^{d+1})_{dn}+\cdots +((-\lambda ax^n)^m)_{dn} = 0
    \] and such an $m$ exists by Lemma \ref{mexists}. 

    Since $\delta$ and $\sigma$ are $k$-linear, \[
        (-\lambda)^d(( ax^n)^d)_{dn}+(-\lambda)^{d+1}(( ax^n)^{d+1})_{dn}+\cdots +(-\lambda)^{m}(( ax^n)^m)_{dn}  = 0.
    \] Since this is true for infinitely many $\lambda$, by a Vandermonde matrix argument, we have that $(( ax^n)^d)_{dn} = a\sigma^n(a)\sigma^{2n}(a)\cdots\sigma^{(d-1)n}(a) = 0.$
\end{proof}

\begin{theorem1}
    Let $R$ be a $k$-algebra where $k$ is an uncountable field, $\sigma$ a locally torsion $k$-algebra automorphism, and $\delta$ a $k$-linear $\sigma$-derivation which is locally nilpotent and $q$-skew. Then $$I = \{a\in R:ax^{t}\in J(R[x;\sigma,\delta])\}$$ is a nil ideal of $R$ for any $t \in \mathbb{N}$.
\end{theorem1}

\begin{proof}
    Let $a \in I$. By Lemma \ref{lemma}, it follows that \[
        a\sigma^n(a)\sigma^{2n}(a)\cdots\sigma^{(d-1)n}(a) = 0,
    \]for any $n \geq t$ with $n > \text{nildeg}(a)$ and some $d \in \mathbb{N}$.
    Since $\sigma$ is locally torsion there exists an $m = \text{tordeg}(a)$ such that $\sigma^m(a) =a$. We can choose $n$ above to be arbitrarily large. If we choose $n$ to be a nonzero multiple of $m$, then \[
    a\sigma^n(a)\sigma^{2n}(a)\cdots\sigma^{(d-1)n}(a) = a^d = 0.
    \]
\end{proof}

Noting that $J(R[x;\sigma,\delta])\cap R$ is an ideal such that $ax^{t} \in J(R[x;\sigma,\delta])$ for every $a \in J(R[x;\sigma,\delta])\cap R$ and every $t \in \mathbb{N}$, this new result partially answers a question of Greenfeld, Smoktunowicz and Ziembowski.

\begin{question}\cite[Question 6.17]{GreenfieldSmoktunowiczZiembowski}
    Consider $J(R[x;\sigma,\delta]) \cap R$. Is it nil if we assume that $\delta$ is locally nilpotent? What if we assume that $\sigma$ is locally torsion?
\end{question}

We are now in a position to prove a $q$-skew Amitsur's Theorem.

\begin{cor}
\label{qskewamitsur}
    Let $R$ be an algebra over an uncountable field of characteristic zero. Let $\sigma$ and $\delta$ be a locally torsion automorphism and a locally nilpotent $q$-skew $\sigma$-derivation on $R$ respectively. Then
    \[
        J(R[x;\sigma,\delta]) = (J(R[x;\sigma,\delta])\cap R)[x;\sigma,\delta],
    \] and $J(R[x;\sigma,\delta])\cap R$ is a nil ideal of $R$.
\end{cor}

\begin{proof}
    By \cite[Theorem 1]{shin2024jacobsonradicalsoreextensions}, we have that \[J(R[x;\sigma,\delta]) = I \cap J(R) +I_0,\]where $I = \{a \in R : ax \in J(R[x;\sigma,\delta])\}$ and $I_0 = \{ \sum_{i\geq 1} a_ix^i : a_i \in I\}$. By Theorem \ref{theorem}, it follows that $I$ is a nil ideal of $R$. Thus, $I \subseteq J(R)$. Hence,\[
        J(R[x;\sigma,\delta]) = I[x;\sigma,\delta].
    \] Thus, $J(R[x;\sigma,\delta]) \cap R = I[x;\sigma,\delta]\cap R = I$ is a nil ideal of $R$ and \[
        J(R[x;\sigma,\delta]) = (J(R[x;\sigma,\delta]) \cap R)[x;\sigma,\delta],
    \]as required.
\end{proof}

To end this paper we suggest some possible directions of study. 

\begin{question}
    Does there exist an Ore extension $R[x;\sigma,\delta]$ satisfying the assumptions of Corollary \ref{qskewamitsur} such that $J(R[x;\sigma,\delta]) \cap R$ is not the nilradical?
\end{question}

\begin{question}
    Does Theorem \ref{theorem} hold if $q\sigma\delta \neq \delta \sigma$? Does Corollary \ref{qskewamitsur} hold if $\sigma$ is only locally finite?
\end{question}

We conjecture that the above question will be answered in the affirmative.

Even in the case of $\sigma = 1$ the following is still open.

\begin{question}
    If $R$ is an algebra over a field of characteristic 0, is $J(R[x;\delta])\cap R$ nil?
\end{question}

\begin{acknowledgements}
    The author is grateful to Professor Agata Smoktunowicz for her assistance, guidance and invaluable advice.
\end{acknowledgements}

\bibliography{References}

@misc{shin2024jacobsonradicalsoreextensions,
  author = {Shin, J.},
  title = {Jacobson radicals of {O}re extensions},
  note = {arXiv:2405.16342},
  year = {2024},
}

@article{GreenfieldSmoktunowiczZiembowski,
  author = {Greenfeld, B. and Smoktunowicz, A. and Ziembowski, M.},
  title = {Five solved problems on radicals of {O}re extensions},
  journal = {Publ. Mat.},
  volume = {63},
  number = {2},
  pages = {423--444},
  year = {2019},
}

@book{lam1991first,
  author = {Lam, T. Y.},
  title = {A First Course in Noncommutative Rings},
  series = {Grad. Texts in Math.},
  publisher = {Springer},
  year = {1991},
}

@book{goodearl2004introduction,
  author = {Goodearl, K. R. and Warfield, R. B.},
  title = {An Introduction to Noncommutative Noetherian Rings},
  publisher = {Cambridge Univ. Press},
  year = {2004},
}

@article{Bedi1980,
  author = {Bedi, S. S. and Ram, J.},
  title = {Jacobson radical of skew polynomial rings and skew group rings},
  journal = {Israel J. Math.},
  volume = {35},
  number = {4},
  pages = {327--338},
  year = {1980},
}

@article{AgataHowFarCanWeGo,
  author = {Smoktunowicz, A.},
  title = {How far can we go with {A}mitsur's conjecture in differential polynomial rings?},
  journal = {Israel J. Math.},
  volume = {219},
  number = {2},
  pages = {555--608},
  year = {2017},
}

@article{SMOKTUNOWICZuncountablefield,
  author = {Smoktunowicz, A. and Ziembowski, M.},
  title = {Differential polynomial rings over locally nilpotent rings need not be {J}acobson radical},
  journal = {J. Algebra},
  volume = {412},
  pages = {207--217},
  year = {2014},
}

@article{Amitsur_1956,
  author = {Amitsur, S. A.},
  title = {Radicals of polynomial rings},
  journal = {Canad. J. Math.},
  volume = {8},
  pages = {355--361},
  year = {1956},
}

@article{Kothe1930,
  author = {K{\"o}the, G.},
  title = {Die Struktur der Ringe, deren Restklassenring nach dem Radikal vollst{\"a}ndig reduzibel ist},
  journal = {Math. Z.},
  volume = {32},
  number = {1},
  pages = {161--186},
  year = {1930},
}

@article{Krempa1972,
  author = {Krempa, J.},
  title = {Logical connections between some open problems concerning nil rings},
  journal = {Fund. Math.},
  volume = {76},
  number = {2},
  pages = {121--130},
  year = {1972},
}

@article{SMOKTUNOWICZ2000427,
  author = {Smoktunowicz, A.},
  title = {Polynomial rings over nil rings need not be nil},
  journal = {J. Algebra},
  volume = {233},
  number = {2},
  pages = {427--436},
  year = {2000},
}

@book{Goodearlformula,
  author = {Goodearl, K. R. and Letzter, E. S.},
  title = {Prime Ideals in Skew and $q$-Skew Polynomial Rings},
  series = {Mem. Amer. Math. Soc.},
  number = {521},
  publisher = {Amer. Math. Soc.},
  year = {1994},
}

@article{DiffPolyRadFerrero,
  author = {Ferrero, M. and Kishimoto, K. and Motose, K.},
  title = {On radicals of skew polynomial rings of derivation type},
  journal = {J. London Math. Soc.},
  volume = {s2-28},
  number = {1},
  pages = {8--16},
  year = {1983},
}
\Addresses
\end{document}